\documentclass[12pt]{article}

\usepackage{color}
\usepackage[utf8]{inputenc}

\usepackage{amsmath}
\usepackage{amssymb}
\usepackage{amsthm}
\usepackage{graphicx}

\DeclareMathOperator{\Co}{Co}

\DeclareMathOperator{\tr}{tr}

\newcommand{\gesem}{\succeq}

\newcommand{\suchthat}{:}
\newcommand{\tran}{^T}
\newcommand{\set}[1]{\{#1\}}

\renewcommand{\Dot}{\bullet}
\newcommand{\intersect}{\cap}

\newcommand{\Abar}{{\bar A}}

\newcommand{\Kcal}{{\cal K}}

\newcommand{\lam}{\lambda}
\newcommand{\BQP}{{\rm BQP}}

\newcommand{\DNN}{{\rm DNN}}
\renewcommand{\Box}{{\rm Box}}
\newcommand{\BoxQP}{{\rm BoxQP}}

\newcommand{\QPB}{{\rm QPB}}

\newcommand{\Pcal}{{\cal P}}
\newcommand{\Rbb}{\mathbb{R}}

\newcommand{\Scal}{{\cal S}}

\newcommand{\Tcal}{{\cal T}}

\newtheorem{lemma}{Lemma}

\begin{document}

\title{Extended Triangle Inequalities for Nonconvex \\
Box-Constrained Quadratic Programming}

\author{ Kurt M. Anstreicher\footnote{Dept. of Business Analytics, University of Iowa, Iowa City, Iowa, USA,
{\tt kurt-anstreicher@uiowa.edu}} \and Diane Puges\footnote{Dept. of Mathematics, Alpen-Adria-Universit{\"a}t, Klagenfurt, Austria, {\tt diane.puges@aau.at}}}

\date{\today}

\maketitle

\begin{abstract}
Let $\Box_n = \set{x\in\Rbb^n\suchthat 0\le x_i\le 1,\ i=1,\ldots,n}$, and let 
$\QPB_n$ denote the convex hull of $\set{(1, x\tran)\tran(1, x\tran)\suchthat x\in\Box_n}$. The quadratic programming problem 
$\min\set{x\tran Q x + q\tran x \suchthat x\in\Box_n}$ where $Q$ is not positive semidefinite (PSD), is equivalent to a linear optimization problem over
$\QPB_n$ and could be efficiently solved if a tractable characterization of $\QPB_n$ was available.  It is known that
$\QPB_2$ can be represented using a PSD constraint combined with constraints generated using the reformulation-linearization technique (RLT). The triangle (TRI) inequalities are also valid for
$\QPB_3$, but the PSD, RLT and TRI constraints together do not fully characterize $\QPB_3$. In this paper we describe new valid linear inequalities for $\QPB_n$, $n\ge 3$ based on strengthening the approximation of $\QPB_3$ given by the PSD, RLT and TRI constraints.  These new inequalities are generated in a systematic way using a known disjunctive characterization for $\QPB_3$. We also describe a conic strengthening of the linear inequalities that incorporates
second-order cone constraints. We show computationally that the new inequalities and their conic strengthenings obtain exact solutions for some nonconvex
box-constrained instances that are not solved exactly using the PSD, RLT and TRI constraints.

\medskip\noindent
{\bf Keywords:} triangle inequalities, box-constrained quadratic programming, nonconvex quadratic programming.\\

\noindent
{\bf Mathematics Subject Classification:} 90C20, 90C26, 90C15.
\end{abstract}

\section{Introduction}

In this paper we are concerned with the box-constrained quadratic programming
problem
\begin{eqnarray*}
\BoxQP:&\max&x\tran Q x + q\tran x\\
&{\rm s.t.}& 0\le x_i\le 1,\quad i=1,\ldots,n.
\end{eqnarray*}
If the matrix $Q$ is positive semidefinite (PSD) then BoxQP can be efficiently solved by a variety of methods, but in general BoxQP is an NP-hard problem that
has been heavily studied in the global optimization literature; see for example \cite{Vandenbussche.Nemhauser.2005} and references therein. BoxQP problems can be approached using general-purpose software such as BARON \cite{BARON}, or specialized methods that are tailored to the problem.  Examples
of the latter include the finite branching algorithm of \cite{Burer.Vandenbussche.2009}, a specialization of \cite{Burer.Vandenbussche.2008} which
itself is a strengthening of the finite branching algorithm of \cite{Vandenbussche.Nemhauser.2005}, and methods based on mixed-integer linear programming 
\cite{Bonami.Gunluk.Linderoth.2018, Yajima.Fujie.1998}. 

Let $\Box_n = \set{x\in\Rbb^n\suchthat 0\le x_i\le 1,\ i=1,\ldots,n}$, and
$\QPB_n=\Co\set{(1, x\tran)\tran(1, x\tran)\suchthat x\in\Box_n}$, where $\Co$ denotes the convex hull. For $x\in\Rbb^n$ and a symmetric $n\times n$ matrix $X$,
let
\[
Y(x,X)= \begin{pmatrix} 1 & x\tran \\ x & X\end{pmatrix}.
\]
The BoxQP problem can then be rewritten in the form of a linear optimization problem
\begin{eqnarray*}
\BoxQP:&\max&Q\Dot X + q\tran x\\
&{\rm s.t.}& Y(x,X)\in\QPB_n,
\end{eqnarray*}
where $Q\Dot X$ denotes the matrix inner product equal to the trace of $QX$. 
Since the extreme points of $\QPB_n$ correspond to rank-one solutions with $X=xx\tran$, and the objective is linear, there will always be an optimal solution
of the latter problem for which $Q\Dot X + q\tran x= x\tran Q x +q\tran x$, the original objective. Rewritten in this second form, the problem of solving
BoxQP becomes the problem of obtaining  a tractable characterization, or approximation, of $\QPB_n$.

In the sequel we will sometimes write $(x,X)\in\QPB_n$ to mean $Y(x,X)\in\QPB_n$ in order to reduce notation.
There are a variety of valid constraints on $Y(x,X)$ that are known for $(x,X)\in\QPB_n$.  Two examples are the PSD condition $Y(x,X)\gesem 0$, and
the constraints obtained by applying the reformulation-linearization technique (RLT) \cite{RLT}, also commonly referred to as the McCormick inequalities,
\begin{equation}\label{eq:RLT}
X_{ij} \le x_i,\quad X_{ij}\le x_j,\quad X_{ij}\ge 0, \quad X_{ij} \ge x_i+x_j-1,\quad 1\le i\le j\le n.
\end{equation}
It is obvious that the PSD condition $Y(x,X)\gesem 0$ implies the diagonal constraint $X_{ii}\ge 0$, and it is easy to show that the PSD condition also
implies $X_{ii}\ge 2x_i-1$.  In the sequel we will be assuming the PSD constraint $Y\gesem 0$, so we will use
DIAG to refer to the diagonal constraints $X_{ii}\le x_i$,  $i=1,\ldots,n$, and RLT to refer to the constraints in \eqref{eq:RLT} for $j>i$ with the addition
of the DIAG constraints.

It was shown in \cite{Anstreicher.Burer.2010} that $\QPB_2$ is exactly represented by the combination of the PSD and RLT constraints, which we will denote by PSD+RLT in the sequel,   and that for larger
$n$ the combination of these constraints usually gives a close approximation of the optimal value, if not the exact optimal value, for test instances of BoxQP problems. It was also shown in
\cite{Anstreicher.Burer.2010} that the PSD+RLT constraints are not sufficient to exactly characterize $\QPB_3$, but that an exact 
disjunctive representation for $\QPB_3$ can be obtained by utilizing a triangulation of $\Box_3$.
Subsequently
\cite{Burer.Letchford.2009}, extending results of \cite{Yajima.Fujie.1998}, showed that any constraint that is valid for the Boolean quadric polytope is also valid for $(x,X)\in\QPB_n$.  In our setting
we consider the Boolean quadric polytope to be the set
$\BQP_n=\Co\set{(1,x\tran)\tran(1, x\tran)\suchthat x_i\in\set{0,1}, i=1,\ldots,n}$. It is then clear that $\BQP_n\subset\QPB_n$. Note that if
$(x,X)\in\BQP_n$ then $X_{ii}=x_i$, $i=1,\ldots,n$, and $X$ is certainly a symmetric matrix. In the literature for polyhedral methods, for example
\cite{Padberg.1989}, the variables for $\BQP_n$ are typically taken to be $x$ and $\set{X_{ij},\ 1\le i<j\le n}$.

There are many known constraints that are valid for $\BQP_n$, including the RLT constraints and the triangle inequalities (TRI)
\begin{equation}\label{eq:TRI}
\begin{array}{rcl}
X_{ij}+X_{ik} &\le& x_i+X_{jk},\\ 
X_{ij}+X_{jk} &\le& x_j+X_{ik},\\
X_{ik}+X_{jk} &\le& x_k+X_{ij},\\
x_i+x_j+x_k &\le& X_{ij}+X_{ik}+X_{jk}+1
\end{array}
\end{equation}
$1\le i<j<k\le n$. It is also known \cite{Padberg.1989} that $\BQP_3$ is fully characterized by the RLT and TRI constraints.  Adding the TRI
constraints to PSD+RLT strengthens the approximation of $\QPB_3$, but it is shown in \cite{Burer.Letchford.2009} that the combination of these constraints, denoted PSD+RLT+TRI in the sequel, is not sufficient to fully characterize $\QPB_3$.

The goal of this paper is to strengthen the approximation of $\QPB_3$ that is given by the PSD+RLT+TRI relaxation.  Our methodology is based on 
extracting information from the disjunctive representation of $\QPB_3$ from \cite{Anstreicher.Burer.2010} to derive additional valid constraints. This
process is described in Section \ref{sec:ETRI1} and results in a total of 24 new valid inequalities on $(x,X)\in\QPB_3$. We refer to these as ETRI1 constraints,
where ``ETRI" stands for ``extended triangle inequality."   In Section \ref{sec:ETRI2/3} we consider a strengthening of the process used in
Section \ref{sec:ETRI1} that results in an additional 24 ETRI2 inequalities and 48 ETRI3 inequalities.  In Section \ref{sec:SOC} we derive conic strengthenings of the ETRI1/2/3 inequalities that utilize second-order cone (SOC) constraints and one additional variable. In Section \ref{sec:comp} we give computational 
results applying the ETRI constraints and their conic strengthenings to a variety of BoxQP instances.  For $n=3$ we show that the conic strengthening of
the ETRI1 constraints obtains the exact solution value for the counterexample from \cite{Burer.Letchford.2009} used to demonstrate that the PSD+RLT+TRI
relaxation does not fully characterize $\QPB_3$.   We also show, based on extensive computations, that the conic strengthening of the ETRI constraints reduces
the worst-case gap for a normalized objective by better than a factor of 2, compared to the PSD+RLT+TRI relaxation of $\QPB_3$. 
 For a set of BoxQP instances
with $5\le n\le 10$, where the PSD+RLT+TRI relaxation is not tight, we show that applying the ETRI constraints and their conic strengthenings reduces the gap and almost always obtains the true optimal solution value.

\noindent{\bf Notation:}  We use $\Scal_n$ to denote symmetric $n\times n$ matrices and $\Scal_n^+$ to denote symmetric positive semidefinite matrices. For $A$ and $B$ both in $S_n$, $A\gesem B$ denotes that $A-B\in \Scal_n^+$ and $A\Dot B$ denotes the trace inner product $A\Dot B = \tr(AB)$.  We use $e$ to denote a vector of arbitrary dimension with each component equal to one and $e_i$ to denote a vector with a 1 in the $i$th coordinate and all other entries
equal to zero. 

\section{Extended triangle inequalities}\label{sec:ETRI1}

In this section we derive new valid inequalities for $\QPB_3$.  The construction of these inequalities is based on a disjunctive representation for $\QPB_3$ given in \cite{Anstreicher.Burer.2010}.  To describe this representation, let $\Tcal_{ijk}\subset
\Box_3$ be the simplex $\set{0\le x_i\le x_j \le x_k\le 1}$. There are 6 such simpleces corresponding to different ordering of the
variables $(x_1,x_2,x_3)$, and these 6 simpleces triangulate $\Box_3$.  For convenience we can put these orderings in lexicographic order as $(123, 132, 213, 231, 312, 321)$, so the first ordering is 123, the fifth is 312, etc. For the $p^{\rm th}$
ordering, corresponding to $ x_i\le x_j \le x_k$, let
$A_p$ be a $3\times 4$ matrix whose columns are the extreme points of $\Tcal_{ijk}$, where the order of the columns in $A_p$ is arbitrary.  For example, a matrix $A_1$ corresponding
to the ordering $0 \le x_1\le x_2\le x_3 \le 1$ is
\[
A_1 = \begin{pmatrix} 1 & 0 & 0 & 0 \\ 1 & 1 & 0 & 0 \\ 1 & 1 & 1 & 0 \end{pmatrix}.
\]
For each such $A_p$ let $\Abar_p$ be the matrix 
\[
\Abar_p = \begin{pmatrix} e\tran \\ A_p \end{pmatrix}.
\]
It is then proved in \cite{Anstreicher.Burer.2010} that
\begin{equation}\label{eq:disjunctive}
\QPB_3  = \left\{ \sum_{p=1}^6 \Abar_p X_p \Abar_p\tran \suchthat X_p\in\DNN_4, e\tran X_p e = \lam_p,  e\tran \lam = 1\right\},
\end{equation}
where $\DNN_4$ are $4\times 4$ matrices that are doubly nonnegative; that is, componentwise nonnegative and 
PSD.  

The representation \eqref{eq:disjunctive} is computable, but cannot be reduced to a system of constraints in the original variables $(x,X)$ due to the constraints $X_p\gesem 0$, $p=1,\ldots,6$. However, by dropping these constraints we obtain a polyhedral set
\begin{equation}\label{eq:P_0}
\Pcal_0=\left\{ \sum_{p=1}^6 \Abar_p X_p \Abar_p\tran \suchthat X_p\ge 0, e\tran X_p e = \lam_p,  e\tran \lam = 1\right\},
\end{equation}
and it is certainly then the case that $\QPB_3\subset \Pcal_0 \intersect \Scal_4^+$. Note that the set $\Pcal_0$ is the convex hull of the union of 6 polyhedral sets of the form
\[
\left\{\Abar_p X_p \Abar_p\tran \suchthat  X_p\ge 0, e\tran X_p e = 1\right\},
\]
each corresponding to an ordering of the variables $\set{x_1, x_2, x_3}$. For any one such ordering, the extreme point matrices of this set are
\begin{equation}\label{eq:extreme_points}
\set{\Abar_p (E_{ij}+E_{ji}) \Abar_p\tran/2\suchthat 1\le i\le j \le 4},
\end{equation}
where $E_{ij}$ is a $4\times 4$ matrix with a one in the $(i,j)$ position and zeros elsewhere.  There are 10 such
extreme points, 4 coming from the diagonal components $i=j$ and 6 from the off-diagonal components $i<j$. There are 9 variables associated with the matrices $Y(x,X)\in\BQP_3$, corresponding to $(x_1, x_2, x_3)$,
$(X_{11}, X_{22}, X_{33})$ and $(X_{12}, X_{13}, X_{23})$, so these 10 extreme points represent a 
simplex in $\Rbb^9$.  It is also possible to give the hyperplane description of this simplex.  To do this, for
a given matrix $\Abar_p$, let $M_p=\Abar_p^{-1}$.  Then for 
\[
Y= \begin{pmatrix} 1 & x\tran \\ x & X \end{pmatrix} = \Abar_p X_p\Abar_p\tran,
\]
 $X_p\ge 0 \iff M_pYM_p \ge 0$, so a hyperplane description of the simplex is given by
\begin{equation}
M_p Y M_p\tran \ge 0,\label{eq:Mhyperplanes}
\end{equation}
For example, for the matrix $\Abar_1$ corresponding to the ordering
$0\le x_1\le x_2\le x_3\le 1$, 
\[
M_1 = \begin{pmatrix} 0  &1 &0& 0\\
              0 &-1& 1& 0\\
              0 &0 &-1& 1\\
              1& 0& 0 &-1\end{pmatrix}.
 \]
 The diagonal constraints from \eqref{eq:Mhyperplanes} are then
 \[
 X_{11} \ge 0,\ X_{11}+X_{22} - 2X_{12} \ge 0,\ X_{22}+X_{33}-2X_{23}\ge 0,\ 1+ X_{33} - 2x_3\ge 0,
 \]
 and these constraints are all implied by the condition $Y\gesem 0$.  The off-diagonal 
constraints from \eqref{eq:Mhyperplanes} are
 \begin{eqnarray*}
 X_{12} - X_{11}\  \ge\ 0, \  X_{13} - X_{12}  &\ge& 0,\ x_1 - X_{13}\ \ge\ 0,\\
 X_{12} + X_{23} - X_{13} - X_{22} &\ge& 0,\\
 x_2 - x_1 +X_{13} - X_{23} &\ge& 0,\\
  x_3 - x_2 +X_{23} - X_{33} &\ge& 0.
  \end{eqnarray*}
 and by direct computation these hyperplanes are exactly the RLT constraints that can be formed
 from the ordering constraints $0\le x_1\le x_2\le x_3\le 1$.
 
 We next set out to obtain a hyperplane description of the set $\Pcal_0$.  To do this, we took the union of the extreme points
 from each of the 6 sets as in \eqref{eq:extreme_points} and then used Polymake  to obtain
 a hyperplane description of the convex hull of this set. This resulted in a total of 72 facet-defining constraints which included all permutation equivalences, to be expected given the form of the disjunction used to construct $\Pcal_0$. The 72 constraints included:
 \begin{itemize}
 \item All of the diagonal and off-diagonal RLT constraints,
 \item The first three of the four triangle inequalities \eqref{eq:TRI},
 \item Nine constraints that were clearly implied by the PSD condition $Y\gesem 0$.
 \end{itemize}
We next checked computationally which constraints were {\em not} dominated by the PSD, RLT and TRI constraints, and found that there were 24 such constraints, three of which are:
 \begin{eqnarray}
2x_1 + X_{11} - 2X_{12} - 2X_{13} + X_{23} &\ge& 0,\nonumber\\
2x_2 - 2X_{12} + X_{13} + X_{22} - 2X_{23} &\ge& 0,\label{eq:ETRI1}\\
2x_3 + X_{12} - 2X_{13} - 2X_{23} + X_{33} &\ge& 0.\nonumber
\end{eqnarray}

Before proceeding, we note that any linear constraint that is valid for $(x,X)\in\QPB_3$ implies other valid constraints that can be constructed by combining
two operations; permuting the indices of variables and switching (or complementing) variables, where switching a variable $x_i$ refers to replacing $x_i$ with $1-x_i$.
Switching $x_i$ results in $X_{ij}$ being replaced by $x_j-X_{ij}$ and similarly switching $x_j$ results in  $X_{ij}$ being replaced by $x_i-X_{ij}$.
Switching both $x_i$ and $x_j$ results in $X_{ij}$ being replaced by $X_{ij}+1 - x_i - x_j$.
The constraints in \eqref{eq:ETRI1} are clearly equivalent under permutation of indices.  Since $3\times 8=24$, one might expect that the remaining 21 constraints corresponded to switching variables in these constraints, but that was not the case. However we determined that the remaining 21 constraints were all dominated by the PSD, RLT and TRI constraints combined with the constraints
in \eqref{eq:ETRI1} and their switchings, a total of 24 additional constraints. We refer to these constraints as ETRI1 constraints, for ``extended triangle inequalities, type 1."  For completeness  we give the coefficients for
all 24 of these ETRI1 constraints in the Appendix. 

We next describe a simple derivation for the constraints in \eqref{eq:ETRI1} that is independent of how we actually found them. For example, to derive the first constraint in \eqref{eq:ETRI1}, consider the valid constraint 
$x_2x_3 \ge x_2+x_3 -1$.  Multiplying both sides by $x_1$ results in the constraint
$x_1(x_2 + x_3 -1) \le x_1 x_2x_3 \le x_2x_3$, which implies the ordinary triangle inequality
$X_{12} + X_{13} \le x_1+X_{23}$. However it is also true that $2x_1 x_2x_3\le x_1^2 +(x_2x_3)^2$, so
\[
2x_1(x_2 + x_3 -1) \le 2x_1x_2x_3 \le x_1^2 +(x_2x_3)^2 \le x_1^2 + x_2x_3,
\]
which implies the ETRI1 constraint $2x_1 +X_{11} -2X_{12} - 2X_{13}  + X_{23} \ge 0$.

Considering  $\set{(x_1, x_2, x_3,X_{11}, X_{22}, X_{33},X_{12}, X_{13}, X_{23})\suchthat (x,X)\in \QPB_3}$ to be a subset of
 $\Rbb^9$, it is known \cite{Burer.Letchford.2009} that this set is full-dimensional, and that faces corresponding to the RLT and TRI constraints being satisfied with equality correspond to facets of dimension 8. In the next lemma we prove that the ETRI1 constraints are tight on faces of dimension 5.

\begin{lemma}\label{lem:ETRI1_dim}
 The set of $\set{x_1, x_2, x_3,X_{11}, X_{22}, X_{33},X_{12}, X_{13}, X_{23}}\subset\Rbb^9$ with $(x,X)\in \QPB_3$ 
 that also satisfy $2x_1 + X_{11} - 2X_{12} - 2X_{13} + X_{23}=0$ has dimension 5.
 \end{lemma}
 
 \begin{proof}
 To characterize the dimension of the face it suffices to consider points where $X_{ij}=x_ix_j$ for all $(i,j)$.
 We first consider the points used in \cite{Burer.Letchford.2009} to prove that the set of feasible points is full dimensional. Of these, six satisfy the additional constraint with equality:
 \begin{itemize}
 \item The point with all variables equal to 0;
 \item The point having $x_2=X_{22}=1$ (respectively $x_3=X_{33}=1$) and all other variables equal to 0;
 \item The point having $x_2=\frac{1}{2}$, $X_{22}=\frac{1}{4}$ (respectively $x_3=\frac{1}{2}$, $X_{33}=\frac{1}{4}$) and
 all other variables equal to 0;
 \item The point having all variables equal to 1.
 \end{itemize}
 Since these 6 points are affinely independent, the face on which the constraint is tight has dimension at least 5. To show that the
 dimension is no higher, consider the equation $2x_1 + x_1^2 - 2x_1x_2 - 2x_1x_3 + x_2x_3=0$ as a quadratic equation
 in $x_1$,
 \begin{equation}\label{eq:x1_quadratic}
 x_1^2 -2x_1(x_2+x_3-1) + x_2x_3 =0.
 \end{equation}
 This equation has a solution if and only if $(x_2+x_3-1)^2\ge x_2x_3$, in which case there are two possible
 values for $x_1$,
 \begin{equation}\label{eq:x1_solutions}
 \begin{array}{rcl}
 x_1^+ &=& x_2+x_3 -1 + \sqrt{(x_2+x_3-1)^2 - x_2x_3},\\
  x_1^- &=& x_2+x_3 -1 - \sqrt{(x_2+x_3-1)^2 - x_2x_3}.
  \end{array}
  \end{equation}
  Assume first that $x_2+x_3\ge 1$.  Since both variables are in $[0,1]$ we have $x_2x_3 \ge x_2+x_3-1\ge 0$,
  so $(x_2+x_3-1)^2 \le (x_2x_3)^2 \le x_2x_3$. We must then have $x_1^+=x_1^-$, and either $x_2x_3=0$ or
  $x_2x_3=1$.  In the first case we then have either $x_2=0, x_3=1$ or $x_2=1, x_3=0$ (implying $x_1=0$), or
  $x_2=x_3=1$, implying $x_1=1$ as well. Assume alternatively that $x_2+x_3 <1$. Then $x_1^-<0$, $x_1^+$ cannot be strictly positive, and to have $x_1^+=0$ requires $x_2x_3=0$, implying either $x_2=0$ or $x_3=0$.  Thus all points
  that satisfy the constraint with equality, other than the point with all variables equal to 1, are in the face with
  $x_1=X_{11}=X_{12}=X_{13}=X_{23}=0$, the last since either $x_2=0$ or $x_3=0$. This face has dimension at most 4, so adding the point with all variables equal to one, the set of points satisfying the constraint with equality can have dimension no greater than 5.
  \end{proof}
  
 Geometrically, the set of $(x_1,x_2,x_3)\in\Box_3$ satisfying \eqref{eq:x1_quadratic} corresponds to the vertex
 $(0,0,0)$, the two adjacent edges with $x_1=x_2=0$ and $x_1=x_3=0$, and the opposite vertex $(1,1,1)$. In $\Box_3$ there are 24 possible choices for 2 adjacent edges, corresponding exactly to the 24 
 ETRI1 constraints obtained from \eqref{eq:ETRI1} and their switchings.  
   
 In addition to the dimension of the face corresponding to a constraint being tight, it is interesting to consider the maximum
 violation of a constraint when it is not imposed.  In Table \ref{tab:violations_I} we consider several cases of increasingly tight
 constraints that contain $\QPB_3$, and the maximum violations of other constraints on these sets.  In the first case, the DIAG constraints are
 added to the PSD condition in order to bound the feasible region. In addition to the
 unscaled maximum violations, we also consider the maximum violations for constraints scaled so that the coefficient
 vector for the variables $(x_1, x_2, x_3,X_{11}, X_{22}, X_{33},X_{12}, X_{13}, X_{23})$ is equal to one. The maximum violations are invariant to switching, but the norm of the coefficient vector may vary.  For example, switching the variable $x_1$ in the RLT constraint $X_{12}\ge 0$ results in the RLT constraint $x_1 - X_{12} \ge 0$. Similarly, switchings of the ETRI1 constraints
\eqref{eq:ETRI1} do not all result in coefficient vectors of the same norm, and the minimum norm is $\sqrt{11}$ rather
than $\sqrt{14}$ for the constraints in \eqref{eq:ETRI1}; see Table \ref{tab:etri1} in the Appendix for details. The normalized violation has a
simple geometric interpretation as the Euclidean distance that a constraint hyperplane which is tight on the relaxed constraint set must be shifted to be tight on $\QPB_3$. 
  
\begin{table}
\begin{center}
\caption{Maximum constraint violations I}
\vspace{.2cm}
\label{tab:violations_I}
\begin{tabular}{r|rrr|rrr}
Enforced& \multicolumn{3}{c|}{Max violation } & \multicolumn{3}{c}{Max normalized} \\ 
Constraints&RLT  & TRI & ETRI1 &RLT & TRI & ETRI1 \\ \hline
PSD+DIAG  & 0.1250  &0.1250 & 0.1250 & 0.1250 & 0.0625 & 0.0377 \\
PSD+RLT & 0 & 0.1250& 0.1111 & 0 & 0.0625 & 0.0335\\
PSD+RLT+TRI & 0 & 0 & 0.0625 & 0 & 0 & 0.0188
\end{tabular}
\end{center}
\end{table}

Finally, note that the derivation of the ETRI1 constraints above was based on a particular triangulation of $\Box_3$; this triangulation is based on
orderings of the variables and has 6 simpleces. It is known \cite{DeLoera.Rambau.Santos.2010} that there are six equivalence classes for triangulations of $\Box_3$, five of which have 6 simpleces and one with only 5 simpleces.  To investigate a possible dependence on the triangulation used, we repeated the derivation of constraints described above using
a triangulation from each of the other 5 equivalence classes. The different triangulations have different permutation and/or switching symmetries, 
and the exact constraints obtained vary with the triangulation used.  However we determined that in each case, all constraints obtained using a different triangulation are implied by the PSD,RLT,TRI and ETRI1 constraints as derived above, and moreover the PSD, RLT and TRI constraints together with the constraints derived from another triangulation, with permutations and switchings, imply the ETRI1 constraints derived above.  We therefore conclude that the
construction of the ETRI1 constraints is independent of the triangulation used. 

 \section{More extended triangle inequalities}\label{sec:ETRI2/3}
 
 The derivation of the ETRI1 inequalities in the previous section was based on dropping the constraints that
 $X_p\gesem 0$ in the disjunctive representation \eqref{eq:disjunctive}, resulting in the set $\Pcal_0$ \eqref{eq:P_0}.  In
 this we consider a strengthening of this procedure based on adding constraints on the $X_p$ matrices from
 \eqref{eq:disjunctive}.  Clearly any constraint of the form $a\tran X_p a\ge 0$ is valid, since we are relaxing the condition that $X_p\gesem 0$. Our choice of constraints to add is based on the form of the extreme point matrices in \eqref{eq:extreme_points}.  As described below \eqref{eq:extreme_points}, 4 of these matrices correspond to diagonal components $i=j$, each resulting in a matrix of the form $\Abar_p E_{ii} \Abar_p\tran$ which is PSD.  The remaining 6 matrices, corresponding to $i<j$, result in matrices of
 the form $\Abar_p (E_{ij}+E_{ji}) \Abar_p\tran/2$, each of which has one positive and one negative eigenvalue, where the latter is due to an eigenvalue of $-1$, with an eigenvector of the form $e_i-e_j$, from the matrix $(E_{ij}+E_{ji})$.  Our idea is to add constraints to $X_p$ which cut off these extreme points by using the eigenvectors corresponding to negative eigenvalues. To this end, let
 \[
  U = \begin{pmatrix}
          1 &-1 &0 & 0\\
          1 & 0 & -1 & 0\\
          1 & 0 & 0 & -1\\
          0 & 1 & -1 & 0\\
          0 & 1 & 0 & -1\\
          0 & 0 & 1 & -1\end{pmatrix}.
          \]
 The rows of $U$ can then be used to generate valid constraints that are violated by the 6 extreme points
 from \eqref{eq:extreme_points} with $i<j$. For each $1\le p\le 6$, the result is then a polyhedral set of the form 
\begin{equation}\label{eq:P_1_sets}
\left\{\Abar_p X_p \Abar_p\tran \suchthat  X_p\ge 0, e\tran X_p e = 1,\ 
u_i\tran X_p u_i \ge 0, i= 1,\ldots, 6\right\},
\end{equation}
where $u_i\tran$ denotes the $i$th row of $U$. Taking the convex combination of the resulting six sets results in a new polyhedral set $\Pcal_1\subset\Pcal_0$,
\begin{equation}\label{eq:P_1}
\Pcal_1 = \left\{ \sum_{p=1}^6 \Abar_p X_p \Abar_p\tran \suchthat  X_p\ge 0, e\tran X_p e = \lam_p, e\tran \lam=1,
u_i\tran X_p u_i \ge 0,\ i= 1,\ldots, 6\right\}.
\end{equation}
 
 To obtain an explicit hyperplane description of $\Pcal_1$ we used a procedure similar to what was done in the previous
 section for $\Pcal_0$. To start, it is not difficult to show that for each $p$, the extreme point matrices from
 \eqref{eq:P_1_sets} are of the form $A_p XA_p$, where $X$ is one of the following $4\times 4$ matrices:
 \begin{itemize}
 \item $X = E_{ii}$, $1\le i\le 4$,
 \item $X= E_{ii}/2 + (E_{ij} + E_{ji})/4$, $1\le i\ne j \le 4$,
 \item $X=E_{ii}/3 + (E_{ij} + E_{ji} + E_{ik} + E_{ki})/6$, $1\le i\ne j < k\ne i\le 4$,
 \item $X= (e_i e\tran + e e_i\tran)/8$, $1\le i\le 4$.
 \end{itemize}
 
There are a total of 32 such extreme matrices for each $p$. We then used Polymake to obtain a hyperplane description of the convex hull of the union of these sets of extreme matrices.
The result of the above procedure was a system of 3723 inequality constraints.  We then checked numerically to see which of these constraints were dominated by the PSD, RLT, TRI and ETRI1 constraints together.  The result was a set of 330 non-dominated constraints, including the following 9 constraints:
 
 \begin{eqnarray}
4x_1 + 4X_{11} - 4X_{12} - 4X_{13} + X_{23} &\ge& 0\nonumber\\     
4x_2 - 4X_{12} + X_{13} + 4X_{22} - 4X_{23} &\ge& 0 \label{eq:ETRI2}\\
4x_3 + X_{12} - 4X_{13} - 4X_{23} + 4X_{33} &\ge& 0  \nonumber  \\
  \nonumber \\
4x_1 + 4X_{11} - 8X_{12} - 4X_{13} + X_{22} + 3X_{23} &\ge& 0\nonumber \\
4x_1 + 4X_{11} - 4X_{12} - 8X_{13} + 3X_{23} + X_{33} &\ge& 0\nonumber \\
4x_2 + X_{11} - 8X_{12} + 3X_{13} + 4X_{22} - 4X_{23} &\ge& 0\label{eq:ETRI3}\\
4x_2 - 4X_{12} + 3X_{13} + 4X_{22} - 8X_{23} + X_{33} &\ge& 0\nonumber \\
4x_3 + X_{11} + 3X_{12} - 8X_{13} - 4X_{23} + 4X_{33} &\ge& 0\nonumber \\
4x_3 + 3X_{12} - 4X_{13} + X_{22} - 8X_{23} + 4X_{33} &\ge& 0\nonumber
\end{eqnarray}

We refer to the constraints in \eqref{eq:ETRI2} and their switchings (a total of 24 constraints) as ETRI2 constraints, and the constraints in 
\eqref{eq:ETRI3} and their switchings (a total of 48 constraints) as ETRI3 constraints.  Finally, we determined that all of the remaining constraints were dominated by the PSD, RLT, TRI, ETRI1, ETRI2 and ETRI3 constraints together. For completeness we
give the coefficients for all of the ETRI2 and ETRI3 constraints in the Appendix.

We next demonstrate the validity of the ETRI2 and ETRI3 constraints independently of how they were derived, as we did for the ETRI1 constraints in the previous section.  Multiplying the constraint $x_2x_3 \ge x_2+x_3-1$ by $4x_1$, we obtain the valid constraint $4x_1x_2+ 4x_1x_3-4x_1 \le 4x_1x_2x_3$. But $(2x_1-x_2x_3)^2 = 4x_1^2 -4x_1x_2x_3 +( x_2x_3)^2\ge 0$, implying $4x_1x_2x_3 \le 4x_1^2+(x_2x_3)^2 \le 4x_1^2 + x_2x_3$. Combining these facts we obtain
$4X_{12}+ 4X_{13}-4x_1 \le 4X_{11} + X_{23}$, which is  the first constraint in
\eqref{eq:ETRI2}. To derive the first constraint in \eqref{eq:ETRI3}, we add $4x_1x_2$ to both sides of 
$4x_1x_2+ 4x_1x_3-4x_1 \le 4x_1x_2x_3$, resulting in $8x_1x_2+ 4x_1x_3-4x_1 \le 4x_1x_2x_3 + 4x_1x_2$. Then
$(2x_1 - x_2(1+x_3))^2 = 4x_1^2 - 4x_1x_2(1+x_3) + x_2^2(1+x_3)^2\ge 0$ implies that
$4x_1x_2x_3 + 4x_1x_2 = 4x_1x_2(1+x_3)\le 4x_1^2 + x_2^2(1+x_3)^2=4x_1^2 + x_2^2+ 2x_2^2x_3 + x_2^2x_3^2
\le 4x_1^2 +x_2^2+3x_2x_3$, and combining these facts we obtain the constraint 
$8X_{12}+ 4X_{13}-4x_1\le 4X_{11} + X_{22}+ 3X_{23}$.

Next we consider the dimensions of the faces of $\QPB_3$ corresponding to the ETRI2 or ETRI3 constraints being tight.

\begin{lemma}\label{lem:ETRI2_dim}
 The set of $\set{x_1, x_2, x_3,X_{11}, X_{22}, X_{33},X_{12}, X_{13}, X_{23}}\subset\Rbb^9$ with $(x,X)\in\QPB_3$ 
 that also satisfy $4x_1 + 4X_{11} - 4X_{12} - 4X_{13} + X_{23}=0$ has dimension 5.
 \end{lemma}
 
 \begin{proof}
 
 The proof is very similar to the proof of Lemma \ref{lem:ETRI1_dim}.  We first give 6 affinely independent points 
 where $X_{ij}=x_ix_j$ for all $(i,j)$ that satisfy the constraint with equality.  These points are identical to the six points used in the proof of Lemma \ref{lem:ETRI1_dim}, except that the point with all variables equal to one is replaced by the point with
 $x_1=X_{12} = X_{13}=\frac{1}{2}$, $X_{11} = \frac{1}{4}$, $x_2=x_3=X_{22} = X_{33} = 1$.
 Since these 6 points are affinely independent, the face on which the constraint is tight has dimension at least 5. To show that the
dimension is no greater than 5, we consider $4x_1 + 4x_1^2 - 4x_1x_2 - 4x_1x_3 + x_2x_3=0$ to be a quadratic equation in 
$x_1$.  The solutions of this equation are then exactly the values of $x_1^+$ and $x_1^-$ in \eqref{eq:x1_solutions}
multiplied by one-half. The remainder of the proof is identical to that of Lemma \ref{lem:ETRI1_dim}, except that the solution
with $x_2=x_3=1$ has $x_1=\frac{1}{2}$ rather than $x_1=1$.
\end{proof}

\begin{lemma}\label{lem:ETRI3_dim}
 The set of $\set{x_1, x_2, x_3,X_{11}, X_{22}, X_{33},X_{12}, X_{13}, X_{23}}\subset\Rbb^9$ with $(x,X)\in\QPB_3$ 
 that also satisfy $4x_1 + 4X_{11} - 8X_{12} - 4X_{13} +X_{22} +3X_{23}=0$ has dimension 4.
 \end{lemma}
 
 \begin{proof}
 
 The proof is similar to the proof of Lemma \ref{lem:ETRI1_dim}, but proving the upper bound for the dimension is more
 complex.  To begin, we give 5 affinely independent points 
 where $X_{ij}=x_ix_j$ for all $(i,j)$ that satisfy the constraint with equality:
 \begin{itemize}
 \item the point with all variables equal to 0;
 \item the point having $x_3=X_{33}=1$ and all other variables equal to 0;
 \item the point having $x_3=\frac{1}{2}$, $X_{33}=\frac{1}{4}$ and all other variables equal to 0;
 \item the point having $x_1=\frac{1}{2}$,  $X_{11} = \frac{1}{4}$, $x_2=X_{22}=1$, $X_{12}=\frac{1}{2}$ and all other variables equal to 0;
 \item the point with all variables equal to 1.
 \end{itemize}
 
 Since these 5 points are affinely independent, the face on which the constraint is tight has dimension at least 4. To show that the
dimension is no greater than 4, we consider $4x_1 + 4x_1^2 - 8x_1x_2 - 4x_1x_3 +x_2^2 +3x_2x_3=0$ to be a quadratic equation in $x_1$.  The possible roots of this quadratic are then
 \begin{eqnarray*}
 x_1^+ &=& \frac{1}{2}\left((2x_2+x_3 -1) + \sqrt{(2x_2+x_3 -1)^2 -x_2(x_2+3x_3)}\right),\\
  x_1^- &=& \frac{1}{2}\left((2x_2+x_3 -1) - \sqrt{(2x_2+x_3 -1)^2 -x_2(x_2+3x_3)}\right).
  \end{eqnarray*}
If $2x_2+x_3<1$ then $x_1^-<0$, and the only solutions with $x_1^+\ge 0$ have $x_2=0$, $x_3\in[0,1]$ and $x_1^+=0$. 
Assume alternatively that $2x_2+x_3\ge 1$.  The discriminant in the expressions for $x_2^+$ and $x_2^-$ can be written as $3x_2^2 + x_2(x_3-4) + (1-x_3)^2$, which we regard as a quadratic in $x_2$.  This quadratic has roots
 \begin{eqnarray*}
 x_2^+ &=& \frac{1}{6}\left((4-x_3) + \sqrt{4+16x_3-11x_3^2}\right),\\
  x_2^- &=& \frac{1}{6}\left((4-x_3) - \sqrt{4+16x_3-11x_3^2}\right),
  \end{eqnarray*}
and it is easy to verify that $4+16x_3-11x_3^2\ge 0$ and also that
$x_2^-\ge 0$ for any $x_3\in[0,1]$. To have $x_2\ge 0$ we then require either $0\le x_2\le x_2^-$ or $x_2^+\le x_2\le 1$. 

Recall that we are assuming that $2x_2+x_3\ge 1$.  Then $x_2\le x_2^-$ implies that $2x_2^-\ge 1-x_3$, which
is equivalent to $15x_3^2-12x_3-3\ge0$.  For $x_3\in[0,1]$ this condition is satisfied only for $x_3=1$, resulting in
$x_2=x_2^-=0$ and $x_1=0$.  Other than this solution, we can then assume that $x_2^+\le x_2\le 1$.  However it is straightforward to show that
$x_2^+\le 1$ is equivalent to $x_3-x_3^2\le 0$, which holds for only $x_3=0$ or $x_3=1$. For $x_3=0$, $x_2=1$ we have $x_1^+=x_1^-=\frac{1}{2}$, and for $x_3=1$, $x_2=1$ we have $x_1^+=1$, $x_1^-=-1$. We have thus shown that the only 
solutions of $4x_1 + 4X_{11} - 8X_{12} - 4X_{13} +X_{22} +3X_{23}=0$ with all variables in $[0,1]$ and
$X_{ij} = x_ix_j$ for all $(i,j)$ are the following:
 \begin{itemize}
 \item points of the form $\set{x_1=x_2=X_{11}=X_{22}=X_{12}=X_{13} = X_{23}=0, x_3\in[0,1], X_{33} = x_3^2}$,
 a set of dimension 2 including the origin;
 \item the point having $x_1=\frac{1}{2}$,  $X_{11} = \frac{1}{4}$, $x_2=X_{22}=1$, $X_{12}=\frac{1}{2}$ and all other variables equal to 0;
 \item the point with all variables equal to 1.
\end{itemize}

\end{proof}

In Table \ref{tab:violations_II} we consider the maximum possible violations of the ETRI2 and ETRI3 constraints when different sets of constraints are imposed. Similar to the presentation in Table \ref{tab:violations_I}, we give both the maximum violations for the constraints as given in \eqref{eq:ETRI2} and \eqref{eq:ETRI3}, and the maximum violations when considering these constraints and their switchings normalized to have coefficient vectors of norm one.  The minimum norm for the vector of coefficients in a switching of the ETRI2 constraints is $\sqrt{50}$, as opposed to $\sqrt{65}$ for the constraints in \eqref{eq:ETRI2} while the 
minimum norm for the vector of coefficients in a switching of the ETRI3 constraints is $\sqrt{115}$, as opposed to $\sqrt{122}$ for the constraints in \eqref{eq:ETRI3}; see Tables \ref{tab:etri2} and \ref{tab:etri3} in the Appendix for details.  

\begin{table}
\begin{center}
\caption{Maximum constraint violations II}
\vspace{.2cm}
\label{tab:violations_II}
\begin{tabular}{r|rr|rr}
Enforced& \multicolumn{2}{c|}{Max violation } & \multicolumn{2}{c}{Max normalized} \\ 
Constraints& ETRI2 & ETRI3 &ETRI2 & ETRI3 \\ \hline
PSD+DIAG  & 0.3333 & 0.3333 &  0.0471 &0.0311  \\
PSD+RLT &0.1111& 0.2038 &0.0157 &0.0190 \\
PSD+RLT+TRI  &0.1005 & 0.1005  & 0.0142 &  0.0094 \\
PSD+RLT+TRI+ETRI1& 0.0856 & 0.0856 & 0.0121 & 0.0080
\end{tabular}
\end{center}
\end{table}

\section{Conic strengthening}\label{sec:SOC}

In this section we describe a conic strengthening of the ETRI constraints derived in the previous two sections.
The conic strengthening is motivated by the argument used to demonstrate validity of the ETRI1 constraints in section \ref{sec:ETRI1}, which is based on the valid constraint 
$x_1x_2 + x_3 -1 \le x_1x_2x_3$ combined with the fact that $2x_1x_2x_3\le x_1^2 + (x_2x_3)^2$.  To strengthen the resulting constraint we introduce one additional variable $z$, which will take the place of the trilinear
term $x_1x_2x_3$.  In the global optimization literature, valid constraints on the variable $z$ are usually projected down to
the set of variables $(x_1,x_2,x_3,z)$. Retaining the variables $X_{12}, X_{13}, X_{23}$, it  can be shown 
\cite{Speakman.Lee.2017} that the convex
hull of $\set{(x_1, x_2, x_3, x_1x_2, x_1x_3, x_2x_3, x_1x_2x_3)\suchthat x\in\Box_3}$ is given by 
$(x_1, x_2, x_3, X_{12}, X_{13}, X_{23}, z)$ that satisfy the following system of linear constraints:
\begin{eqnarray}
z\ \ge\ 0,\ z & \le& X_{12},\ z\ \le\ X_{13},\ z\ \le\ X_{23}\nonumber\\
X_{12} +X_{13} &\le& x_1 + z, \ X_{12} +X_{23}\ \le\ x_2 + z, \ X_{13} +X_{23}\ \le\ x_3 + z, \label{eq:trilinear}\\
x_1+x_2+x_3+z &\le& X_{12} + X_{13} +X_{23} +1\nonumber
\end{eqnarray}
The constraints in \eqref{eq:trilinear} can be viewed as extensions of the ordinary RLT and TRI constraints on
$(x_1, x_2, x_3, X_{12}, X_{13}, X_{23})$. 
Note that the constraints in \eqref{eq:trilinear} do not involve the diagonal variables $X_{11}, X_{22}, X_{33}$. However, since $z$ is a proxy for $x_1x_2x_3$, where $x\in\Box_3$, and $(x_2x_3)^2\le x_2x_3$, the following constraints are also valid:
\begin{equation}\label{eq:SOC}
z^2 \le X_{11}X_{23},\ \ z^2 \le X_{22}X_{13},\ \ z^2 \le X_{33}X_{12}.
\end{equation}
The constraints in \eqref{eq:SOC} are rotated second-order-cone (SOC) constraints that can be imposed in addition to the
constraints from \eqref{eq:trilinear}.  These SOC constraints can also be imposed on switchings of variables, where switchings of variables are applied to $z$ in the obvious way; for example, switching $x_1$ results in a switched value for $z$ of $(1-x_1)x_2x_3 = x_2x_3 - x_1x_2x_3 = X_{23} - z$.

\begin{lemma}\label{lem:SOC_implies_ETRI}
The constraints \eqref{eq:trilinear} together with the constraints \eqref{eq:SOC} and their switchings imply all of the
ETRI1 and ETRI2 constraints.
\end{lemma}

\begin{proof}
It suffices to show that the SOC constraint $z^2 \le X_{11}X_{23}$ together with the constraints in \eqref{eq:trilinear}
imply the ETRI1 constraint $2x_1 + X_{11} - 2X_{12} - 2X_{13} + X_{23}\ge 0$
and the ETRI2 constraint $4x_1 + 4X_{11} - 4X_{12} - 4X_{13} + X_{23} \ge 0$.
From \eqref{eq:trilinear} we have $X_{12} +X_{13} -x_1 \le z$, and $z^2 \le X_{11}X_{23}$ implies that
$z\le (X_{11} + X_{23})/2$ from the arithmetic-geometric mean inequality.  Combining these two facts we obtain
$2X_{12} +2X_{13} -2x_1 \le X_{11} + X_{23}$, which is exactly the required ETRI1 constraint.  The same SOC 
constraint, written as $4z^2\le 4X_{11}X_{23}$, also implies that $2z \le (4X_{11} + X_{23})/2$, or
$4z \le 4X_{11} + X_{23}$.  Combining that inequality with $X_{12} +X_{13} -x_1 \le z$ results in
$4X_{12} +4X_{13} -4x_1 \le 4X_{11} + X_{23}$, which is exactly the required ETRI2 constraint.
\end{proof}

Although the constraints from \eqref{eq:trilinear} and \eqref{eq:SOC} with their switchings imply the ETR1 and ETRI2 constraints, it turns out that these strengthened constraints have no effect on the maximum violation for the ETRI3 constraints, which remains 0.0856 as reported in Table \ref{tab:violations_II}. We next show that it is also possible to give a conic strengthening of the 
ETRI3 constraints.  This strengthening is based on the fact that the rank-one matrix
\[
\begin{pmatrix} 1 & x_1 & x_2 + x_2x_3 \\
                      x_1 &   x_1^2   & x_1x_2 + x_1x_2x_3 \\
                      x_2 + x_2x_3  & x_1x_2 + x_1x_2x_3 & x_2^2 + 2x_2^2x_3 + x_2^2x_3^2
\end{pmatrix}
\]
 is PSD, which together with $x_2^2x_3^2\le x_2^2x_3\le x_2x_3$  implies the valid rotated SOC constraint
 \begin{equation}\label{eq:ETRI3_SOC}
 (X_{12}+z)^2 \le X_{11}(X_{22} + 3X_{23}).
 \end{equation}
 
 \begin{lemma}
 The constraints \eqref{eq:trilinear} together with the SOC constraints obtained from \eqref{eq:ETRI3_SOC} by permuting indices and switching variables imply all of the ETRI3 constraints.
 \end{lemma}
 
 \begin{proof} It suffices to show that the constraints \eqref{eq:trilinear} together with  \eqref{eq:ETRI3_SOC} imply the
 ETRI3 constraint  $ 4x_1 + 4X_{11} - 8X_{12} - 4X_{13} + X_{22} + 3X_{23} \ge 0$.   From \eqref{eq:trilinear}
   we have  $4X_{12} +4X_{13} -4x_1 \le 4z$, and adding $4X_{12}$ to both sides obtains  the inequality
        $8X_{12} +4X_{13} -4x_1 \le 4X_{12}+4z$. From \eqref{eq:ETRI3_SOC} we have $4(X_{12}+z)^2 \le
        4X_{11}(X_{22} + 3X_{23})$, which implies that $2(X_{12}+z)\le (4X_{11} + X_{22} + 3X_{23})/2$ from the arithmetic-geometric mean inequality.  Combining these facts, we obtain  $8X_{12} +4X_{13} -4x_1 \le  4X_{11} + X_{22}+ 3X_{23}$, which is exactly the required ETRI3 inequality. \end{proof}        
        
 \section{Computational results}\label{sec:comp}
        
 In this section we report a variety of different computational results obtained when implementing the constraints described in Sections
 \ref{sec:ETRI1}, \ref{sec:ETRI2/3} and \ref{sec:SOC}.  We will begin with results for problems over $\QPB_3$, and then consider instances
 over $\QPB_n$ for $n>3$. All problems in this section were solved using the Mosek or SeDuMi interior-point solvers running under Matlab or Julia.
 
The following example from Burer and Letchford \cite{Burer.Letchford.2009} shows that the PSD, RLT and TRI conditions together are not sufficient to characterize $\QPB_3$.
\[
{\rm BL:} \max\left\{x\tran Q x + q\tran x \suchthat x\in \Box_3\right\},
\]
where
\[
Q = \begin{pmatrix} -2.25&-3& -3 \\ -3& 0 &-0.5 \\ -3 &-0.5& 1 \end{pmatrix},\quad
q = \begin{pmatrix} 3 \\ 1 \\ 0 \end{pmatrix}.
\]
In particular, the exact solution value for the BL problem is 1.0, but the value using the relaxation of $\QPB_3$ that imposes the PSD, RLT and TRI constraints is approximately 1.09291. 

In Table \ref{tab:BL} we give the values obtained by solving the BL problem over increasingly tight relaxations of $\QPB_3$. The first row in the table corresponds to the PSD+RLT+TRI relaxation, and subsequent rows are labeled by the constraints in addition to PSD+RLT+TRI.
For example, adding the  ETRI1/2/3 constraints to the PSD+RLT+TRI relaxation reduces the gap to the true solution value from 0.09291 to 0.05882. In the last case, ``SOC" refers to the constraints obtained from \eqref{eq:SOC} by switching variables as well as from \eqref{eq:ETRI3_SOC} by permuting indices and switching variable, and in this case the original RLT and TRI
 constraints are replaced by the extended system \eqref{eq:trilinear}. The resulting SOC
 strengthening of the ETRI constraints obtains the exact solution value 1.0, and in fact this value is attained by adding only one constraint which
 is a switching of one of the constraints \eqref{eq:SOC}. To our knowledge, this is the first time that the BL problem has been solved exactly without the use of spatial branching, dynamically generated cutting planes \cite{Dong.Anstreicher.2013}, or an extended-variable formulation such as the exact disjunctive formulation or the formulation from a hierarchy of cones $\Kcal_n^r$ that better approximate the copositive or completely positive cone for $r>0$ \cite{Pena.Vera.Zuluaga.2007, Dong.2013}.  
 
 Although adding the SOC constraints to the PSD+RLT+TRI relaxation obtains the true optimal value for the BL problem, the solution does not
 immediately provide a feasible $x$ with $x\tran Qx+q\tran x=1$.  The reason for this is that the BL problem has multiple optimal solutions, and in this
 case an interior-point solver such as Mosek or SeDuMi will not converge to a rank-one solution.  This deficiency can be overcome by an additional step that
 imposes a constraint $Q\Dot X + q\tran x = 1$ and re-solves the problem with a random objective, which then generates a rank-one solution having
 $x\tran Qx+q\tran x=1$.

\begin{table}
\begin{center}
\caption{Objective values for Burer-Letchford problem}
\vspace{.2cm}
\label{tab:BL}
\begin{tabular}{r|r}
Enforced&  Objective \\
Constraints&   value\\  \hline
PSD+RLT+TRI  & 1.09291\\
+ETRI1&1.06613\\
+ETRI1/2/3&1.05882\\
+SOC&1.00000
\end{tabular}
\end{center}
\end{table}  
       
For the BL example, the vector of objective coefficients for the variables $(x_1, x_2, x_3,X_{11},$ $X_{22}, X_{33},X_{12}, X_{13}, X_{23})$
is $(3,1,0,-2.25,0,1,-6,-6,-1)$, and the norm of this coefficient vector is approximately 9.4373.  If the coefficient vector is normalized to have
norm 1, then the gap of 0.09291 for the original problem corresponds to a gap of  0.009845 for the normalized objective. 
We next consider the largest possible gap for a normalized objective when minimized over a system of constraints compared to the exact minimum
over  $\QPB_3$, where the latter can be computed using the disjunctive representation.  Maximizing this difference is a non-convex problem, but it can be approximated by repeatedly generating random coefficients.  A coefficient vector that tentatively maximizes the difference can also be potentially improved by making smaller perturbations to it. We performed extensive computations in an effort to find normalized coefficients that approximately maximize the gap for different relaxations of $\QPB_3$. The results of these computations are reported in Table \ref{tab:max_gaps}.  For the PSD+DIAG constraint set the maximum gap appears to be achieved for an objective corresponding to the normalized sum of 2 RLT constraints, for example $X_{12} + X_{13}$, and for the PSD+RLT constraint set the maximum appears to be achieved for objective coefficients corresponding to a normalized triangle inequality TRI. For the PSD+RLT+TRI relaxation, the maximum gap appears to be achieved by normalizing an ETRI1 inequality with coefficient vector of norm $\sqrt{11}$, as in Table \ref{tab:violations_I}. However, when the ETRI1 constraints are added to the PSD+RLT+TRI relaxation the maximum normalized gap does not
correspond to an ETRI2 or ETRI3 constraint, as can be seen by comparing the maximum value of 0.0135 in Table \ref{tab:max_gaps} to the max normalized values in the last row of Table \ref{tab:violations_II}.
As in Table \ref{tab:BL}, the rows below PSD+RLT+TRI are labeled using the constraints that are added to the PSD+RLT+TRI relaxation.  As reported in the table, these computations indicate that the use of the SOC tightenings of the ETRI constraints reduce the maximum gap
for a normalized objective by better than a factor of 2 compared to the PSD+RLT+TRI relaxation. 
 
\begin{table}
\begin{center}
\caption{Maximum gaps over relaxations}
\vspace{.2cm}
\label{tab:max_gaps}
\begin{tabular}{r|rrrr}
Enforced&  Max gap& \\
Constraints&   (normalized) & Remark\\ \hline
PSD+DIAG   & 0.1768 &Sum of 2 RLT \\
PSD+RLT &  0.0625&TRI \\
PSD+RLT+TRI  & 0.0188&  ETRI1\\ \hline
+ETRI1& 0.0135\\
+ETRI1/2/3&0.0111\\
+SOC&0.0086
\end{tabular}
\end{center}
\end{table}  
  
We next consider results on problems with $n>3$.  Like the original TRI constraints, for $n>3$  the ETRI constraints immediately extend to any triple of indices
$1\le i<j<k\le n$.  The same is true for the SOC tightenings described in Section \ref{sec:SOC}, with the added consideration that any such triple
requires an additional variable $z_{ijk}$ in place of the variable $z$ used in Section \ref{sec:SOC}.  When solving a problem with $n>3$, we add RLT,
TRI and ETRI constraints in several ``rounds," limiting the number of violated constraints on each round before re-solving the problem.  This strategy prevents the addition of an excessive and ultimately unnecessary number of constraints.  We use a similar strategy when incorporating SOC constraints,
and in that case add additional variables $z_{ijk}$ and the corresponding linear constraints \eqref{eq:trilinear} only as needed to enforce violated SOC
conditions.

A test set of 54 BoxQP instances first used in \cite{Vandenbussche.Nemhauser.2005} has been used in several subsequent papers \cite{Anstreicher.2012,
Bonami.Gunluk.Linderoth.2018,Burer.Vandenbussche.2009} to compare the performance
of different methods. These problems have dimensions from 20 to 60 and were generated with varying degrees of sparsity for $Q$ and $q$. It was shown in \cite{Anstreicher.2012} that the PSD+RLT+TRI relaxation is tight for all but one of these 54 problems.  We first considered adding the ETRI and SOC constraints on that one problem (instance 50-050-1), but obtained no improvement in the bound. We next generated additional instances with similar dimensions using the same methodology employed in \cite{Vandenbussche.Nemhauser.2005}. As expected, the PSD+RLT+TRI relaxation was tight on a high
proportion of these problems and obtained a rank-one solution with objective value equal to the relaxation bound.  For the occasional instance where this did not occur, we obtained no improvements using the ETRI constraints or their SOC tightenings.  This was disappointing but not completely unexpected given
the rarity of instances with any gap and the fact that these problems have a substantial space of off-diagonal variables.  As shown in \cite{Burer.Letchford.2009}, any constraint that is valid for the Boolean Quadric Polytope (BQP)  is also valid for the variables $x$ and $\set{X_{ij} \suchthat j>i}$.  Such valid constraints include the RLT and TRI constraints, but for $n>3$ there are large families of additional valid constraints for the BQP
\cite{Boros.Hammer.1993,Padberg.1989}.

We next considered generating problems similar to those from \cite{Vandenbussche.Nemhauser.2005}, but with smaller $n$.  For $5\le n\le 10$, and
problems of the form $\max\set{x\tran Qx + q\tran x\suchthat x\in\Box_n}$, we generated $Q$ and $q$ as follows. For a given density parameter $d$  between 0 and 100,
each $q_i$ and $Q_{ij}$, $i<j$ is an integer randomly distributed on [-50,50] with probability $d$\% and is otherwise zero, and $Q_{ji}=Q_{ij}$.  As was the case for larger problems, the PSD+RLT+TRI relaxation was tight for a high proportion of instances generated in this way.  However, for instances where
the PSD+RLT+TRI relaxation was not tight, the ETRI constraints and their SOC strengthenings almost always  improved the bound.  In Table 
\ref{tab:n>3_results} we give results for 12 such problem instances\footnote{Data for the problems in Table
  \ref{tab:n>3_results} is available from the authors.}.  The instances are labeled using the same scheme ($n$-$d$-\#) as in 
  \cite{Vandenbussche.Nemhauser.2005}. In the table, P+R+T corresponds to the PSD+RLT+TRI relaxation, and the other columns correspond to additional constraints. For each relaxation we give both the relaxation value, an upper bound on the optimal value, and the feasible objective value
  $x\tran Q x + q\tran x$ from the solution $(x,X)$, a lower bound on the optimal value. In Table \ref{tab:n>3_gaps} we give the gaps between
 the relaxation value and optimal value, and relaxation value and feasible value, for the same problems.
 
 \begin{table}
\begin{center}
\caption{Objective values on problems with $n>3$}
\vspace{.2cm}
\scriptsize
\label{tab:n>3_results}
\begin{tabular}{r|r|rrrr|rrrr}  
 & & \multicolumn{4}{c}{Relaxation objective value}&\multicolumn{4}{c}{Feasible objective value}\\
 Instance	&	OPT	&	P+R+T	&	+ETRI1	&	+ETRI1/2/3	&	+SOC	&	P+R+T	&	+ETRI1	&	+ETRI1/2/3	&	+SOC	\\ \hline
05-050-1	&	0.0061	&	0.2122	&	0.0090	&	0.0090	&	0.0090	&	-7.7721	&	-1.1640	&	-1.1353	&	-1.1598	\\
06-070-1	&	1.0000	&	1.0607	&	1.0050	&	1.0050	&	1.0000	&	-6.3672	&	-5.7869	&	-5.7870	&	1.0000	\\
06-080-1	&	0.0000	&	0.0356	&	0.0000	&		&		&	-3.2871	&	0.0000	&		&		\\
07-060-1	&	5.0000	&	5.1521	&	5.0784	&	5.0784	&	5.0000	&	-7.9104	&	-6.6276	&	-6.6276	&	5.0000	\\
08-075-1	&	130.3906	&	130.6563	&	130.4851	&	130.4794	&	130.3906	&	94.2010	&	97.7381	&	96.9528	&	130.3906	\\
08-075-2	&	49.0000	&	49.1094	&	49.0354	&	49.0000	&		&	32.9629	&	38.0247	&	49.0000	&		\\
08-080-1	&	39.7189	&	40.0934	&	39.7189	&		&		&	21.1042	&	39.7189	&		&		\\
08-085-1	&	80.0000	&	80.2355	&	80.0000	&		&		&	27.2094	&	80.0000	&		&		\\
09-070-1	&	8.0788	&	8.6110	&	8.3477	&	8.0788	&		&	-9.6524	&	-10.6659	&	8.0788	&		\\
09-080-1	&	163.0347	&	163.3233	&	163.2550	&	163.2550	&	163.0347	&	144.3430	&	145.0280	&	145.0280	&	163.0347	\\
10-070-1	&	289.0000	&	289.3192	&	289.0000	&		&		&	247.8940	&	289.0000	&		&		\\
10-075-1	&	159.2132	&	160.0067	&	159.9102	&	159.8997	&	159.2132	&	118.2380	&	118.9980	&	119.8370	&	159.2132	
  \end{tabular}
\end{center}
\end{table}

\begin{table}[h]
\begin{center}
\caption{Objective gaps on problems with $n>3$}
\vspace{.2cm}
\scriptsize
\label{tab:n>3_gaps}
\begin{tabular}{r|r|rrrr|rrrr}  
 & & \multicolumn{4}{c}{Relaxation value to optimal value}&\multicolumn{4}{c}{Relaxation value to feasible value}\\
Instance	&	OPT	&	P+R+T	&	+ETRI1	&	+ETRI1/2/3	&	+SOC	&	P+R+T	&	+ETRI1	&	+ETRI1/2/3	&	+SOC	\\ \hline
05-050-1	&	0.0061	&	0.2061	&	0.0029	&	0.0029	&	0.0029	&	7.9843	&	1.1730	&	1.1442	&	1.1688	\\
06-070-1	&	1.0000	&	0.0607	&	0.0050	&	0.0050	&	0.0000	&	7.4278	&	6.7919	&	6.7920	&	0.0000	\\
06-080-1	&	0.0000	&	0.0356	&	0.0000	&		&		&	3.3227	&	0.0000	&		&		\\
07-060-1	&	5.0000	&	0.1521	&	0.0784	&	0.0784	&	0.0000	&	13.0625	&	11.7060	&	11.7060	&	0.0000	\\
08-075-1	&	130.3906	&	0.2656	&	0.0945	&	0.0887	&	0.0000	&	36.4553	&	32.7470	&	33.5266	&	0.0000	\\
08-075-2	&	49.0000	&	0.1094	&	0.0354	&	0.0000	&		&	16.1465	&	11.0107	&	0.0000	&		\\
08-080-1	&	39.7189	&	0.3745	&	0.0000	&		&		&	18.9892	&	0.0000	&		&		\\
08-085-1	&	80.0000	&	0.2355	&	0.0000	&		&		&	53.0261	&	0.0000	&		&		\\
09-070-1	&	8.0788	&	0.5322	&	0.2689	&	0.0000	&		&	18.2633	&	19.0136	&	0.0000	&		\\
09-080-1	&	163.0347	&	0.2886	&	0.2202	&	0.2202	&	0.0000	&	18.9803	&	18.2270	&	18.2270	&	0.0000	\\
10-070-1	&	289.0000	&	0.3192	&	0.0000	&		&		&	41.4252	&	0.0000	&		&		\\
10-075-1	&	159.2132	&	0.7935	&	0.6970	&	0.6865	&	0.0000	&	41.7687	&	40.9122	&	40.0627	&	0.0000	
\end{tabular}
\end{center}
\end{table}

 For these 12 problems, the ETRI1 constraints obtain the optimal value on 4 instances, the ETRI1/2/3 constraints obtain the optimal value
 on 2 instances, and the SOC strengthenings of the ETRI1/2/3 constraints obtain the optimal value on 5 instances. For the remaining instance (05-50-01), adding the ETRI1 constraints obtains a substantial reduction in the relaxation bound, but the ETRI2/3 constraints and SOC strengthenings give no further improvement. The true optimal value for this instance was computed using Gurobi. For the 11 problems where the gap was closed to zero, in 10 instances the relaxed problem had a rank-one solution that provided a feasible $x$ with $x\tran Q x+q\tran x$ equal to the bound value. In the remaining case (07-060-1) the solution is not rank-one due to the presence of multiple optimal solutions, and an optimal $x$ can be generated using the same technique described for the BL problem earlier in the section. It is worthwhile to note that although the gaps between the  PSD+RLT+TRI relaxation value and the optimal value for the problems in Table \ref{tab:n>3_results} are already quite small, the gaps to the feasible solution from the PSD+RLT+TRI relaxation are much larger. In addition, the feasible values may worsen as the relaxation is tightened, as can be seen in several cases. Other methods such as local search could be used in an attempt to improve these feasible values, but we did not investigate this possibility.
Finally, although the potential number of ETRI and SOC constraints is quite large ($11,520$ ETRI constraints for $n=10$), the number of ETRI or SOC constraints used in the solutions of the problems in Table \ref{tab:n>3_results} was very small. For problems solved to optimality using ETRI constraints, the number or ETRI constraints used was never more than 15, and for problems solved to optimality using SOC constraints, the number of SOC constraints used was at most 6.

\section*{Acknowlegements}

Work on this paper was begun when one of the authors (Anstreicher) was visiting the Dept. of Mathematics at the University of Klagenfurt as a 
Karl-Popper-Fellow. Support
from the University of Klagenfurt is gratefully acknowleged. This research was funded in whole or in part by the Austrian Science Fund (FWF) [10.55776/DOC78]. For open access purposes, the author has applied a CC BY public copyright license to any author-accepted manuscript version arising from this submission.
        
\bibliographystyle{abbrv}
\bibliography{QPB}       

 \section*{Appendix}
    
 In the tables below we give coefficients for the ETRI1 constraints \eqref{eq:ETRI1} and their switchings, the ETRI2 constraints   \eqref{eq:ETRI2} and their
 switchings and the  ETRI3 constraints   \eqref{eq:ETRI3} and their switchings.  In each row the coefficients $c=(c_1, c_2,c_3)$ and
 $C=(C_{11}, C_{22}, C_{33}, C_{12}, C_{13}, C_{23})$ and constant $b$ correspond to a constraint given in the form
 \[
 \sum_{i=1}^3 c_ix_i + \sum_{i=1}^3\sum_{j\ge i} C_{ij}X_{ij} + b \ge 0.
 \]
Each group of 8 constraints corresponds to one of the constraints from \eqref{eq:ETRI1}, \eqref{eq:ETRI2} or \eqref{eq:ETRI3}
followed by switchings of that constraint.

\begin{table}[h]
\begin{center}
\caption{Coefficients for ETRI1 constraints}
\vspace{.2cm}
\label{tab:etri1}
\scriptsize
\begin{tabular}{rrrrrrrrrr}  
$x_1$	&	$x_2$	&	$x_3$	&	$X_{11}$	&	$X_{22}$	&	$X_{33}$	&	$X_{12}$	&	$X_{13}$	&	$X_{23}$	& $b$	\\ \hline
2	&	0	&	0	&	1	&	0	&	0	&	-2	&	-2	&	1	&	0	\\
0	&	1	&	0	&	1	&	0	&	0	&	-2	&	2	&	-1	&	0	\\
0	&	0	&	1	&	1	&	0	&	0	&	2	&	-2	&	-1	&	0	\\
-2	&	-1	&	-1	&	1	&	0	&	0	&	2	&	2	&	1	&	1	\\
-4	&	-2	&	-2	&	1	&	0	&	0	&	2	&	2	&	1	&	3	\\
-2	&	-1	&	2	&	1	&	0	&	0	&	2	&	-2	&	-1	&	1	\\
-2	&	2	&	-1	&	1	&	0	&	0	&	-2	&	2	&	-1	&	1	\\
0	&	1	&	1	&	1	&	0	&	0	&	-2	&	-2	&	1	&	0	\\ \hline
0	&	2	&	0	&	0	&	1	&	0	&	-2	&	1	&	-2	&	0	\\
1	&	0	&	0	&	0	&	1	&	0	&	-2	&	-1	&	2	&	0	\\
-2	&	-4	&	-2	&	0	&	1	&	0	&	2	&	1	&	2	&	3	\\
-1	&	-2	&	2	&	0	&	1	&	0	&	2	&	-1	&	-2	&	1	\\
0	&	0	&	1	&	0	&	1	&	0	&	2	&	-1	&	-2	&	0	\\
-1	&	-2	&	-1	&	0	&	1	&	0	&	2	&	1	&	2	&	1	\\
2	&	-2	&	-1	&	0	&	1	&	0	&	-2	&	-1	&	2	&	1	\\
1	&	0	&	1	&	0	&	1	&	0	&	-2	&	1	&	-2	&	0	\\ \hline
0	&	0	&	2	&	0	&	0	&	1	&	1	&	-2	&	-2	&	0	\\
-2	&	-2	&	-4	&	0	&	0	&	1	&	1	&	2	&	2	&	3	\\
1	&	0	&	0	&	0	&	0	&	1	&	-1	&	-2	&	2	&	0	\\
-1	&	2	&	-2	&	0	&	0	&	1	&	-1	&	2	&	-2	&	1	\\
0	&	1	&	0	&	0	&	0	&	1	&	-1	&	2	&	-2	&	0	\\
2	&	-1	&	-2	&	0	&	0	&	1	&	-1	&	-2	&	2	&	1	\\
-1	&	-1	&	-2	&	0	&	0	&	1	&	1	&	2	&	2	&	1	\\
1	&	1	&	0	&	0	&	0	&	1	&	1	&	-2	&	-2	&	0	
\end{tabular}
\end{center}
\end{table}   
    
\begin{table}
\begin{center}
\caption{Coefficients for ETRI2 constraints}
\vspace{.2cm}
\label{tab:etri2}
\scriptsize
\begin{tabular}{rrrrrrrrrr}  
$x_1$	&	$x_2$	&	$x_3$	&	$X_{11}$	&	$X_{22}$	&	$X_{33}$	&	$X_{12}$	&	$X_{13}$	&	$X_{23}$	& $b$	\\ \hline
4	&	0	&	0	&	4	&	0	&	0	&	-4	&	-4	&	1	&	0	\\
0	&	1	&	0	&	4	&	0	&	0	&	-4	&	4	&	-1	&	0	\\
0	&	0	&	1	&	4	&	0	&	0	&	4	&	-4	&	-1	&	0	\\
-4	&	-1	&	-1	&	4	&	0	&	0	&	4	&	4	&	1	&	1	\\
-12	&	-4	&	-4	&	4	&	0	&	0	&	4	&	4	&	1	&	8	\\
-8	&	-3	&	4	&	4	&	0	&	0	&	4	&	-4	&	-1	&	4	\\
-8	&	4	&	-3	&	4	&	0	&	0	&	-4	&	4	&	-1	&	4	\\
-4	&	3	&	3	&	4	&	0	&	0	&	-4	&	-4	&	1	&	1	\\ \hline
0	&	4	&	0	&	0	&	4	&	0	&	-4	&	1	&	-4	&	0	\\
1	&	0	&	0	&	0	&	4	&	0	&	-4	&	-1	&	4	&	0	\\
-4	&	-12	&	-4	&	0	&	4	&	0	&	4	&	1	&	4	&	8	\\
-3	&	-8	&	4	&	0	&	4	&	0	&	4	&	-1	&	-4	&	4	\\
0	&	0	&	1	&	0	&	4	&	0	&	4	&	-1	&	-4	&	0	\\
-1	&	-4	&	-1	&	0	&	4	&	0	&	4	&	1	&	4	&	1	\\
4	&	-8	&	-3	&	0	&	4	&	0	&	-4	&	-1	&	4	&	4	\\
3	&	-4	&	3	&	0	&	4	&	0	&	-4	&	1	&	-4	&	1	\\ \hline
0	&	0	&	4	&	0	&	0	&	4	&	1	&	-4	&	-4	&	0	\\
-4	&	-4	&	-12	&	0	&	0	&	4	&	1	&	4	&	4	&	8	\\
1	&	0	&	0	&	0	&	0	&	4	&	-1	&	-4	&	4	&	0	\\
-3	&	4	&	-8	&	0	&	0	&	4	&	-1	&	4	&	-4	&	4	\\
0	&	1	&	0	&	0	&	0	&	4	&	-1	&	4	&	-4	&	0	\\
4	&	-3	&	-8	&	0	&	0	&	4	&	-1	&	-4	&	4	&	4	\\
-1	&	-1	&	-4	&	0	&	0	&	4	&	1	&	4	&	4	&	1	\\
3	&	3	&	-4	&	0	&	0	&	4	&	1	&	-4	&	-4	&	1
\end{tabular}
\end{center}
\end{table}         

\begin{table}
\begin{center}
\caption{Coefficients for ETRI3 constraints}
\vspace{.2cm}
\label{tab:etri3}
\scriptsize
\begin{tabular}{rrrrrrrrrr}  
$x_1$	&	$x_2$	&	$x_3$	&	$X_{11}$	&	$X_{22}$	&	$X_{33}$	&	$X_{12}$	&	$X_{13}$	&	$X_{23}$	& $b$	\\ \hline
4	&	0	&	0	&	4	&	1	&	0	&	-8	&	-4	&	3	&	0	\\
0	&	3	&	0	&	4	&	1	&	0	&	-8	&	4	&	-3	&	0	\\
-4	&	-2	&	3	&	4	&	1	&	0	&	8	&	-4	&	-3	&	1	\\
-8	&	-5	&	-3	&	4	&	1	&	0	&	8	&	4	&	3	&	4	\\
-12	&	-8	&	-4	&	4	&	1	&	0	&	8	&	4	&	3	&	8	\\
-8	&	-5	&	4	&	4	&	1	&	0	&	8	&	-4	&	-3	&	4	\\
-4	&	6	&	-1	&	4	&	1	&	0	&	-8	&	4	&	-3	&	1	\\
0	&	3	&	1	&	4	&	1	&	0	&	-8	&	-4	&	3	&	0	\\ \hline
4	&	0	&	0	&	4	&	0	&	1	&	-4	&	-8	&	3	&	0	\\
-4	&	3	&	-2	&	4	&	0	&	1	&	-4	&	8	&	-3	&	1	\\
0	&	0	&	3	&	4	&	0	&	1	&	4	&	-8	&	-3	&	0	\\
-8	&	-3	&	-5	&	4	&	0	&	1	&	4	&	8	&	3	&	4	\\
-12	&	-4	&	-8	&	4	&	0	&	1	&	4	&	8	&	3	&	8	\\
-4	&	-1	&	6	&	4	&	0	&	1	&	4	&	-8	&	-3	&	1	\\
-8	&	4	&	-5	&	4	&	0	&	1	&	-4	&	8	&	-3	&	4	\\
0	&	1	&	3	&	4	&	0	&	1	&	-4	&	-8	&	3	&	0	\\ \hline
0	&	4	&	0	&	1	&	4	&	0	&	-8	&	3	&	-4	&	0	\\
3	&	0	&	0	&	1	&	4	&	0	&	-8	&	-3	&	4	&	0	\\
-8	&	-12	&	-4	&	1	&	4	&	0	&	8	&	3	&	4	&	8	\\
-5	&	-8	&	4	&	1	&	4	&	0	&	8	&	-3	&	-4	&	4	\\
-2	&	-4	&	3	&	1	&	4	&	0	&	8	&	-3	&	-4	&	1	\\
-5	&	-8	&	-3	&	1	&	4	&	0	&	8	&	3	&	4	&	4	\\
6	&	-4	&	-1	&	1	&	4	&	0	&	-8	&	-3	&	4	&	1	\\
3	&	0	&	1	&	1	&	4	&	0	&	-8	&	3	&	-4	&	0	\\ \hline
0	&	4	&	0	&	0	&	4	&	1	&	-4	&	3	&	-8	&	0	\\
3	&	-4	&	-2	&	0	&	4	&	1	&	-4	&	-3	&	8	&	1	\\
-4	&	-12	&	-8	&	0	&	4	&	1	&	4	&	3	&	8	&	8	\\
-1	&	-4	&	6	&	0	&	4	&	1	&	4	&	-3	&	-8	&	1	\\
0	&	0	&	3	&	0	&	4	&	1	&	4	&	-3	&	-8	&	0	\\
-3	&	-8	&	-5	&	0	&	4	&	1	&	4	&	3	&	8	&	4	\\
4	&	-8	&	-5	&	0	&	4	&	1	&	-4	&	-3	&	8	&	4	\\
1	&	0	&	3	&	0	&	4	&	1	&	-4	&	3	&	-8	&	0	\\ \hline
0	&	0	&	4	&	1	&	0	&	4	&	3	&	-8	&	-4	&	0	\\
-8	&	-4	&	-12	&	1	&	0	&	4	&	3	&	8	&	4	&	8	\\
3	&	0	&	0	&	1	&	0	&	4	&	-3	&	-8	&	4	&	0	\\
-5	&	4	&	-8	&	1	&	0	&	4	&	-3	&	8	&	-4	&	4	\\
-2	&	3	&	-4	&	1	&	0	&	4	&	-3	&	8	&	-4	&	1	\\
6	&	-1	&	-4	&	1	&	0	&	4	&	-3	&	-8	&	4	&	1	\\
-5	&	-3	&	-8	&	1	&	0	&	4	&	3	&	8	&	4	&	4	\\
3	&	1	&	0	&	1	&	0	&	4	&	3	&	-8	&	-4	&	0	\\ \hline
0	&	0	&	4	&	0	&	1	&	4	&	3	&	-4	&	-8	&	0	\\
-4	&	-8	&	-12	&	0	&	1	&	4	&	3	&	4	&	8	&	8	\\
3	&	-2	&	-4	&	0	&	1	&	4	&	-3	&	-4	&	8	&	1	\\
-1	&	6	&	-4	&	0	&	1	&	4	&	-3	&	4	&	-8	&	1	\\
0	&	3	&	0	&	0	&	1	&	4	&	-3	&	4	&	-8	&	0	\\
4	&	-5	&	-8	&	0	&	1	&	4	&	-3	&	-4	&	8	&	4	\\
-3	&	-5	&	-8	&	0	&	1	&	4	&	3	&	4	&	8	&	4	\\
1	&	3	&	0	&	0	&	1	&	4	&	3	&	-4	&	-8	&	0
\end{tabular}
\end{center}
\end{table}

\end{document}